\documentclass[psamsfonts,reqno,11pt]{amsart}
\usepackage{graphicx}
\usepackage{graphics}
\usepackage{latexsym}
\usepackage{amsmath}
\usepackage{amssymb}
\usepackage[numbers,sort&compress]{natbib}

\newtheorem{theorem}{Theorem}[section]
\newtheorem{lemma}[theorem]{Lemma}
\newtheorem{proposition}[theorem]{Proposition}
\newtheorem{corollary}[theorem]{Corollary}

\theoremstyle{definition}
\newtheorem{definition}[theorem]{Definition}
\newtheorem{example}[theorem]{Example}

\theoremstyle{remark}
\newtheorem{remark}[theorem]{Remark}
\newtheorem{problem}[theorem]{Problem}

\numberwithin{equation}{section}

\newcommand\R{\mathbb{R}}

\def\sideremark#1{\ifvmode\leavevmode\fi\vadjust{\vbox
to0pt{\vss \hbox to 0pt{\hskip\hsize\hskip1em
\vbox{\hsize2cm\tiny\raggedright\pretolerance10000
\noindent#1\hfill}\hss}\vbox to8pt{\vfil}\vss}}}

\begin{document}

\title[A Note on The Mazur-Ulam property of almost-CL-spaces]{A note on The Mazur-Ulam property of almost-CL-spaces}

\author{Dongni Tan}
\author{Rui Liu$^*$}

\address{Department of Mathematics, Tianjin University of Technology, Tianjin 300384, P.R. China}

\email{tandongni0608@gmail.com}

\address{Department of Mathematics and LPMC, Nankai
University, Tianjin 300071, P.R. China}


\email{ruiliu@nankai.edu.cn}



\thanks{
$*$Corresponding author.\\
\indent The first author was supported by the Natural Science Foundation of China 11126249. The second author was supported by the Natural Science Foundation of China 11001134 and 11126250, the Fundamental Research Funds for the
Central Universities, and the Tianjin
Science \& Technology Fund 20100820.}

\begin{abstract}We introduce the (T)-property, and prove that every Banach space with the (T)-property has the Mazur-Ulam property (briefly MUP). As its immediate applications, we obtain that almost-CL-spaces admitting a smooth point (specially, separable almost-CL-spaces) and a two-dimensional space whose unit sphere is a hexagon has the MUP. Furthermore, we discuss the stability of the spaces having the MUP by the $c_0$- and $\ell_1$-sums, and show that the space $C(K,X)$ of the vector-valued continuous functions has the
the MUP, where $X$ is a separable almost-CL-space and $K$ is a compact metric space.
\end{abstract}

\date{\today}
\maketitle

\section{Introduction}
The classical Mazur-Ulam theorem states that every surjective
isometry between normed spaces is a linear mapping up to
translation. In 1972, P. Mankiewiz \cite{16} extended this by showing that every surjective
isometry between the open connected subsets of normed spaces can be
extended to a surjective affine isometry on the whole space. This result implies that the
metric structure on the unit ball of a real normed space constraints the linear structure
of the whole space. It is of interest to us whether this result can be
extended to unit spheres. In 1987, D. Tingley \cite{Ti} first
studied isometries on the unit sphere and raised the \emph{isometric
extension problem}:
\begin{problem}\label{pr:1}
Let $X$ and $Y$ be normed spaces with the unit spheres $S_X$ and
$S_Y$, respectively. Assume that $T : S_X\to S_Y$ is a surjective
isometry. Does there exist a linear isometry
$\widetilde{T} : X\to Y$ such that $\widetilde{T}|_{S_X}=T$?
\end{problem}

\noindent In \cite{Ti}, D. Tingley only proved that $T(-x)=-T(x)$ for all $x\in S_X$, with $X$ being assumed to be a finite-dimensional Banach space.
During the past decade, Ding and his students have been
working on this topic and have obtained many important results (see
\cite{D8,D9,LL} for a survey of results). Cheng and Dong \cite{CD} introduced the notion of the following Mazur-Ulam property and raised a generalized Mazur-Ulam question, whether every Banach space admits the Mazur-Ulam property, which is equivalent to the isometric extension problem.
\begin{definition}
A Banach space $X$ is said to have the \emph{Mazur-Ulam property} (briefly, \emph{MUP}) provided that for every Banach space $Y$, every surjective isometry $T$ between the two unit spheres of $X$ and $Y$ is the restriction of a linear isometry between the two spaces.
\end{definition}
\noindent It has been proved that several classical Banach spaces have the MUP, as for example
$\ell^{p}(\Gamma),\,L^p(\mu)$($1\leq p\leq \infty$) and $C(\Omega)$ (see [4--6, 14--16, 23--25]). However, the peoblem is still open in the general case, even in two dimensions.

Now, a natural and important problem is that of finding an essential
sphere structure such that the isometric extension problem holds.
Recently, L. Cheng and Y. Dong \cite{CD} attacked the problem for the class of CL-spaces
admitting a smooth point and polyhedral spaces. Unfortunately
their interesting attempt failed by a mistake at the very end of the proof.
L. Cheng told to V. Kadets and M. Martin that they do not see how their proof
can be repaired (also see the introduction in \cite{KM}). Kadets and Martin \cite{KM} prove that finite-dimensional
polyhedral Banach spaces have the MUP.

In this note, we introduce the (T)-property, and prove that every Banach space with the (T)-property has the MUP. As its applications, we obtain that almost-CL-spaces admitting a smooth point (specially, separable almost-CL-spaces) and a two-dimensional space whose unit sphere is a hexagon has the MUP.
Furthermore, we discuss the stability of the spaces having the MUP by the $c_0$-sums and $\ell_1$-sums, and show that the space $C(K,X)$ of the vector-valued continuous functions has
the MUP, where $X$ is a separable almost-CL-space and $K$ is a compact metric space.
This is the first time to study the stability of the MUP and the first examples of vector-valued spaces having the MUP.
We cite some known results on almost-CL-space theory to show that our result
includes many important examples of Banach spaces.

Throughout this paper, we consider the spaces all over the real
field. For a Banach space $X$,  $B_X, S_X$ and $ex(X)$ will stand for the unit ball of $X$, the unit sphere of $X$ and the set of extreme
points of $B_X$, respectively.  Let $x\in S_X$, the
\emph{star of $x$ with respect to $S_X$} is
defined by
$St(x)=\big\{\,y : y\in S_X,\, \|\,y+x\,\|=2
\,\big\}.$ If $C$ is a convex subset of $S_X$, we say that $C$ is
a \emph{maximal convex subset of $S_X$} if it is not properly
contained in any other convex subset of $S_X$ (cf. \cite{Ti}).
\begin{definition}
A Banach space $X$ is said to be \emph{a (an almost-) CL-space} if for every maximal convex set $C$ of $S_X$, we have $$B_X=co(C\cup-C)\quad (B_X=\overline{co}(C\cup-C), \mbox{respectively}).$$
\end{definition}
\noindent The notion of CL-spaces were first introduced
in 1960 by R. Fullerton \cite{F}. Almost-CL-spaces first studied without a name in the memoir by J. Lindenstrauss \cite {J} and were further extended   by Lima \cite{Lima1,Lima2} who showed that $L_1(\mu)$ spaces and their preduals are CL-spaces.  For
general information on CL-spaces and almost-CL-spaces, we refer to \cite{MR1,MR2}.

\section{Main Results}

First, we introduce the notion of D.Tingley property (briefly, (T)-property) using the
notion of stars and maximal convex subsets of unit spheres.

\begin{definition}\label{de:1}
A Banach space $X$ is said to have the (T)-property if there exists a subset $\{x_\gamma:\gamma\in\Gamma\}$ of $S_X$
satisfying the following conditions:
\begin{enumerate}
\item[(i)] For any $\gamma\in\Gamma$,
$ St(x_\gamma)$ is a maximal convex subset of $S_X$;

\item[(ii)] $S_X=\bigcup\{ St(x_\gamma):\gamma\in\Gamma\}$;

\item[(iii)] For every $x\in S_X$, $\gamma\in\Gamma$ and $\varepsilon>0$
there exist $x^+_\gamma\in St(x_\gamma)$ and
$x^-_\gamma\in St(-x_\gamma)$ such that
\[\|x-x^+_\gamma\|+\|x-x^-_\gamma\|\leq 2+\varepsilon.\]
\end{enumerate}
\end{definition}

\begin{remark}
Note that when $E$ is a subspace of a Banach space $X$ and $C$ is a maximal convex subset of $S_X$, $C_E=C\cap E$ may not be a maximal
convex set of $S_E$. For example, in the $3$-dimensional space
$X=l_\infty^{\,(3)}$, $C=\{(x,y,1)\,|\,|x|\le 1,|y|\le 1\}$ is a
maximal convex set of $S_X$. Let
$E=\mathrm{span}\{(1,1,1),(1,-1,0)\}$. Then $C_E=C\cap
E=\{(1,1,1)\}$. Indeed, if
$(x,y,1)=\alpha\cdot(1,1,1)+\beta\cdot(1,-1,0)=(\alpha+\beta,\alpha-\beta,\alpha)$
with $|x|\le1$ and $|y|\le1$, then we have $x=y=\alpha=1$ and
$\beta=0$. But $\{(1,1,1)\}$ is not a maximal convex set of $S_E$,
since $(1,1,1)\in \mathrm{convex}\{(1,1,1),(1,-1,0)\}\subset S_E$.
\end{remark}

\begin{remark}\label{re:1} Note that, for any $x\in S_X$, if $ St(x)$ is
convex, then it is a maximal convex subset of
$S_X$. Actually, let $C\subset S_X$ be convex with $x\in C$. Then for every $y\in C$,  $(x+y)/2\in C$ implies that
$\|x+y\|=2$. Thus $y \in St(x)$, and hence $C\subset  St(x)$.
\end{remark}
This remark and \cite[Corollary 1, Lemma 2]{liu07a} show the following proposition.
\begin{proposition}\label{LT1}
Let $X$ and $Y$ be Banach spaces, and let $T: S_X\rightarrow S_Y$ be a surjective isometry. Then for every $x\in S_X$,
the following two statements are equivalent:
\begin{enumerate}
\item[(i)]$St(x)$ is a maximal convex subset of $S_X$.
\item[(ii)] $T(St(x))$ is a maximal convex subset of $S_Y$.
\end{enumerate}

\end{proposition}
The following is our main theorem.
\begin{theorem}\label{th:1}
Let $X$ and $Y$ be Banach spaces, and let $T: S_X\rightarrow S_Y$ be a surjective isometry. If $X$ has the (T)-property, then
$T$
can be extended to a linear surjective isometry from $X$ onto $Y$.
\end{theorem}
\begin{proof}
Let $\{x_\gamma:\gamma\in\Gamma\}\subset S_X$ satisfy the conditions (i)-(iii) of (T)-property.
We first show that there exists a family of pairs
$\{x_\gamma^*,y_\gamma^*\}_{\gamma\in\Gamma}\subset S(X^*)\times
S(Y^*)$
such that
\begin{equation}\label{LT4}
y_\gamma^*(Tx)=x_\gamma^*(x)=\theta
\end{equation} for all
$\gamma\in\Gamma$, $\theta\in\{\pm1\}$ and\; $x\in
 St(\theta x_\gamma)$.
 Indeed, since $St(x_\gamma)$ is a maximal
convex subset of $S_X$, it follows from the Eidelheit separation theorem that there exists
an $x_\gamma^*\in X^*$ with $\|x_\gamma^*\|=1$ such that
$x_\gamma^*(x)=1$ for all $x\in St(x_\gamma)$, and hence
$x_\gamma^*(x)=-1$ for all $x\in St(-x_\gamma)$ since $St(-x_\gamma)=-St(x_\gamma)$.

Note from \cite[Corollary 1]{liu07a} that \[T( St(-x_\gamma))= St(T(-x_\gamma))=
- St(-T(-x_\gamma))=-T( St(x_\gamma)).\]
Then by Proposition \ref{LT1}, a similar argument applied to $T(St(x_\gamma))$ yields the desired result.

We next show that

\begin{equation}\label{LT:2}y_\gamma^*\circ
T|_{S_X}=x_\gamma^*|_{S_X} \quad \forall \gamma\in \Gamma .\end{equation}

For every $x\in S_X, \gamma\in\Gamma$ and $\varepsilon>0$, by the (T)-property (iii) there exist
$x_\gamma^+\in St(x_\gamma)$ and
$x_\gamma^-\in St(-x_\gamma)$ such that
$\|x-x_\gamma^+\|+\|x-x_\gamma^-\|\leq2+\varepsilon $.  Then
\begin{eqnarray*}\label{eq:1}2&=&y_\gamma^*(Tx_\gamma^+)-y_\gamma^*(Tx_\gamma^-)\\
&=&y_\gamma^*(Tx_\gamma^+-Tx)+y_\gamma^*(Tx-Tx_\gamma^-)\nonumber\\
&\leq&\|Tx_\gamma^+-Tx\|+\|Tx-Tx_\gamma^-\|\\&=&\|x_\gamma^+-x\|+\|x-x_\gamma^-\|\leq2+\varepsilon,
\end{eqnarray*}
which implies that
\[y_\gamma^*(Tx_\gamma^+-Tx)\geq \|x_\gamma^+-x\|-\varepsilon , \ \ y_\gamma^*(Tx-Tx_\gamma^-)\geq\|x-x_\gamma^-\|-\varepsilon.\]
This and (\ref{LT4}) thus give
\begin{eqnarray*}y_\gamma^*(Tx)&=&1-y_\gamma^*(Tx_\gamma^+-Tx)\\&\leq&
1-\|x_\gamma^+-x\|+\varepsilon\le1-x_\gamma^*(x_\gamma^+-x)+\varepsilon\\&\leq&x_\gamma^*(x)+\varepsilon,
\end{eqnarray*}
and
\begin{eqnarray*}y_\gamma^*(Tx)&=&-1+y_\gamma^*(Tx-Tx_\gamma^-)\\&\geq &
-1+\|x-x_\gamma^-\|-\varepsilon\geq-1+x_\gamma^*(x-x_\gamma^-)-\varepsilon\\&=&x_\gamma^*(x)-\varepsilon.
\end{eqnarray*}
It follows that $y_\gamma^*(Tx)=x_\gamma^*(x)$. That is
$y_\gamma^*\circ T|_{S_X}=x_\gamma^*|_{S_X}$.

Note that $S_X=\bigcup\{ St(x_\gamma):\gamma\in\Gamma\}$ and $S_Y=\bigcup\{ T(St(x_\gamma)):\gamma\in\Gamma\}$.
Thus  we can deduce from this and (\ref{LT:2}) that  for all $x_1,x_2\in S_X$ and $\lambda\in\R$,
\begin{eqnarray}\label{LT5}\|x_1-\lambda x_2\|&=&
\max_{\gamma\in\Gamma}|x_\gamma^*(x_1-\lambda x_2)|\nonumber\\
&=&\max_{\gamma\in\Gamma}|y_\gamma^*(Tx_1-\lambda Tx_2)|\nonumber\\
&=&\|Tx_1-\lambda Tx_2\|.
\end{eqnarray}

Now we may define the required extension $\widetilde{T}$ of $T$ by
$$\widetilde{T}(x)= \left
\{ \begin{array}{ll}
\|x\|T(\frac{x}{\|x\|}), & \mbox{ if }\;x\neq0;\\
0, & \mbox{ if }\;x=0.
\end{array}
\right.$$
It is easily checked from (\ref{LT5}) that $\widetilde{T}: X \rightarrow Y$ is a surjective isometry whose restriction to the unit sphere $S_X$ is just
$T$. The Mazur-Ulam theorem hence shows that $\widetilde{T}$ is linear as required. The proof is complete.
\end{proof}

We next shall that every almost-CL-space $X$ admitting a smooth point has the (T)-property.
For this statement, we need one more lemma.

\begin{lemma}\label{tan1}
Let $X$ be a Banach space, and let $C$ be a maximal convex subset of $S_X$. If $x\in C$ is a smooth point, then $C=St(x)$.
\begin{proof}
It is trivial that $C\subset  St(x)$. For the converse, given $y\in St(x)$, the Hahn-Banach theorem implies that there exists a functional $f$ in $S_{X^*}$ such that $$f(x+y)=\|x+y\|=2.$$ Therefore $f(x)=f(y)=1$. Note from the Hahn-Banach and Eidelheit separation theorems that there exists a functional $g$ in $S_{X^*}$ such that $g^{-1}\{1\}\cap S_X=C$. Since $x$ is a smooth point, we have $g=f$. Thus $y\in g^{-1}\{1\}\cap S_X=C$, and hence $St(x)\subset C$.
\end{proof}
\end{lemma}
We recall from \cite{F,Lima1, MR1} some basic facts and some notions on the almost-CL-space $X$. By applying the Hahn-Banach and Krein-Milman theorems, one can
easily prove that every maximal convex subset $F$ of $S_X$ has the form
\begin{eqnarray*}
F=\{x\in S_X: x^*(x)=1\}
\end{eqnarray*}
for some $x^*\in ex(B_{X^*})$, denoted by $F_{x^*}$. We denote by $mexB_{X^*}$ the set of those $x^*\in ex(B_{X^*})$
such that $F_{x^*}$ is a maximal convex subset of $S_X$. We say that an almost-CL-space $X$ admits a smooth point if every maximal convex set $C$ of its unit sphere $S_X$ has a smooth point.
\begin{proposition}\label{pr:1}
 Every almost-CL-space $X$ admitting a smooth point (in particular, every separable almost-CL-space) has the (T)-property.
\end{proposition}
\begin{proof}
By Zorn's Lemma and the previous considerations we see that
\begin{eqnarray*}
S_X=\bigcup \{F_{x^*}: x^*\in mexB_{X^*}\}.
\end{eqnarray*}
For every $x^*\in mexB_{X^*}$, let $x\in F_{x^*}$ be a smooth point of $S_X$. Then it follows directly from Lemma \ref{tan1} that
\begin{eqnarray}\label{t1}
 St(x)=F_{x^*} .
\end{eqnarray}

Now for each  $x\in S_X, x^*\in  mexB_{X^*}$ and $\varepsilon>0$, since $X$ is an almost CL-space, there exists a $y$ in co$(F_{x^*},-F_{x^*}$) such that $$\|y-x\|<\varepsilon/2.$$ For this $y$, choose $\{\lambda_k\}_{k=1}^{n}\subset\mathbb{R^+}\cup\{0\}$ with $\sum_{k=1}^{n}\lambda_k=1$, $\{x_k\}_{k=1}^{i}\subset F_{x^*}$ and $\{x_k\}_{k=i+1}^{n}\subset -F_{x^*}$ such that
\begin{eqnarray*}
y=\sum_{k=1}^{n}\lambda_kx_k.
\end{eqnarray*}
We may assume that $\sum_{k=1}^{i}\lambda_k\neq 0 \,\mbox{or}\, 1$. Set $\lambda=\sum_{k=1}^{i}\lambda_k, y_1=(1/\lambda)\sum_{k=1}^{i}\lambda_kx_k$ and $y_2=1/(1-\lambda)\sum_{k=i+1}^{n}\lambda_kx_k$. Then it is easy to see that
\begin{eqnarray*}
y=\lambda y_1+(1-\lambda)y_2
\end{eqnarray*}
with $y_1\in F_{x^*}$ and $y_2\in -F_{x^*}$. Therefore
\begin{eqnarray*}
\|y-y_1\|+\|y-y_2\|=\|y_1-y_2\|=2.
\end{eqnarray*}
Thus $\|x-y_1\|+\|x-y_2\|\leq 2+\varepsilon.$
\end{proof}

The next statement is an immediate application of Theorem \ref{th:1} and Proposition \ref{pr:1}.
\begin{corollary}\label{cor:1}
Every almost-CL-space $X$ admitting a smooth point (in particular, every separable almost-CL-space) has the MUP.
\end{corollary}
We have the following important examples of spaces having the MUP as immediate consequences of Corollary \ref{cor:1} and results on almost-CL-spaces (see \cite{MR1}).
\begin{example}
For any nonempty index set $\Gamma$, the space $\mathcal{L}^{\infty}(\Gamma)$-type spaces \cite {liu07a} including $c_0(\Gamma)$, $c(\Gamma)$ and $\ell^\infty(\Gamma)$ are CL-spaces admitting a smooth point, and thus they have the MUP.
\end{example}

\begin{example}
 Let $(\Omega,\Sigma,\mu)$ be a $\sigma$-finite measure space. Then $L_1(\Omega,\Sigma,\mu)$ is a separable CL-space, and thus it has the MUP.
\end{example}

\begin{example}
For every compact Hausdorff space $K$, $C(K)$ is a CL-space. If $K$ is a completely regular and its singleton point is a $G_\delta$, then $C(K)$ admits a smooth point, and thus it has the MUP.  In particular, for every compact metric space $K$, $C(K)$ has the MUP.
\end{example}

We will show that some two-dimensional spaces with hexagonal norms which are clearly not almost-CL-spaces having the MUP. These spaces are firstly introduced by M. Martin and J. Meri \cite{MR3}.

\begin{example}
The space $X_\gamma=(\mathbb{R}^2,\|\cdot \|_{1/2})$ whose norm is given by
\begin{eqnarray*}
\|(\xi,\eta)\|_{1/2}=\max\{|\eta|, |\xi|+1/2|\eta|\}, \quad \forall\, (\xi,\eta)\in X_{1/2}
\end{eqnarray*}
has the (T)-property, and thus it has the MUP.
\begin{proof}
Let $x_1=(0,1), x_2=(3/4,1/2),x_3=(-3/4,1/2)$, and let
\begin{eqnarray*}
A=\{\pm x_1, \pm x_2, \pm x_3\}.
\end{eqnarray*}
Then it is easily checked that for every $x\in A$, $ St(x)$ is a maximal convex set of $S_{X_{1/2}}$, and
\begin{eqnarray*}
S_{X_{1/2}}=\bigcup_{x\in A} St(x).
\end{eqnarray*}

For every $x\in S_{X_{1/2}}$, replace $x$ by $-x$ if necessary, we may assume that $$x\in \bigcup_{i=1}^{3} St(x_i).$$ To prove that $X_{1/2}$ satisfies the condition (iii) of the (T)-Property, we distinguish three cases.

Case 1. $x\in  St(x_1)$.  Set $x_2^+=(1/2, 1), x_2^-=(-1,0).$ Then $x_2^+ \in  St(x_2)$ and $x_2^-\in St(-x_2).$
The definition of the norm thus implies that
\begin{eqnarray*}
\|x-x_2^+\|+\|x-x_2^-\|=2.
\end{eqnarray*}
A similar conclusion holds with $x_2^+$ replaced by $x_3^+$ and $x_2^-$ replaced by $x_3^-$ respectively where $x_3^+=(-1/2, 1), x_3^-=(1,0)$ satisfy  $x_3^+ \in  St(x_3)$ and $x_3^-\in\mbox{star}(-x_3).$

Case 2. $x\in  St(x_2)$. Set $x_1^+=(1/2, 1), x_1^-=(1/2,-1).$  Then it is easy to check that
 \begin{eqnarray*}
\|x-x_i^+\|+\|x-x_i^-\|=2 \quad \forall i=1,3.
\end{eqnarray*}

Case 3. $x\in  St(x_3)$. This case yields to a similar argument as Case 2.
\end{proof}
\end{example}

Now Corollary \ref{cor:1} combined with the results \cite[section 4 ]{MR1} on the stability of  the classes of
almost-CL-spaces by $c_0$- and $\ell_1$-sums will show the following result.
\begin{proposition}
Let $\{X_i : i\in I\}$ be a family of separable almost-CL-spaces, and let
$X = [\bigoplus_{i\in I} X_i]_{c_0}$ or $X = [\bigoplus_{i\in I} X_i]_{\ell_1}$.
Then $X$ has the MUP.
\end{proposition}

We shall consider the spaces $C(K,X)$ of continuous functions
from a compact Hausdorff space $K$ into a Banach space $X$. These are also the first examples of vector valued spaces having the MUP.

\begin{proposition}
Let $K$ be a compact metric space $K$, and let $X$ be a separable almost-CL-space. Then
$C(K, X)$ has the MUP.
\end{proposition}
\begin{proof}
It follows from \cite{Vi} that $C(K, X)$ is a separable Banach space. Note that the result in \cite[section 5 ]{MR1} states that $C(K, X)$ is an almost-CL-space if and only if $X$ is. The desired conclusion follows immediately from Corollary \ref{cor:1}.
\end{proof}

\vskip8mm

\noindent\textbf{Acknowledgment.} The authors express their appreciation to Guanggui Ding
for many very helpful comments regarding isometric theory in Banach spaces.

\vskip1.2cm

\end{document}